\documentclass[12pt,a4paper,leqno]{amsart}

\input amssymb.sty

\newtheorem{thm}{Theorem}
\newtheorem{theorem}{Theorem}
\newtheorem*{theorem*}{Theorem}
\newtheorem{lemma}[thm]{Lemma}
\newtheorem*{thm*}{Theorem}
\newtheorem*{thmA*}{Theorem A}
\newtheorem*{thmB*}{Theorem B}
\newtheorem*{TheoremCFZ1*}{Theorem CFZ1}
\newtheorem*{MartyTheorem*}{Marty's Theorem}

\newtheorem*{TheoremCFZ2*}{Theorem CFZ2}
\newtheorem*{TheoremCFZ3*}{Theorem CFZ3}
\newtheorem*{theoremB*}{Theorem B}
\newtheorem*{theoremA*}{Theorem A}
\newtheorem*{theoremPL2*}{Theorem PL2}
\newtheorem*{TheoremR*}{Theorem R}
\newtheorem*{TheoremGN*}{Theorem GN}
\newtheorem*{theoremLu*}{Theorem Lu}

\newtheorem*{cor*} {Corollary}

 \newtheorem*{defn*} {Definition}

\newtheorem*{ZL*}{Zalcman's Lemma}
\newtheorem*{Zp*}{ZP1 Lemma}
\newtheorem*{ZP*}{ZP2 Lemma}
\newtheorem*{Ne*}{N Lemma}
\newtheorem*{claim*}{Claim}

  \theoremstyle{definition}

\newtheorem*{examp*}{Example}
\newtheorem*{remark*}{Remark}

 \theoremstyle{remark}

\newtheorem*{notation*}{Notation}

\newcommand{\C}{\mathbb{C}}
\newcommand{\D}{\displaystyle}
\newcommand{\DF}[2]{\frac{\D#1}{\D#2}}

 \renewcommand{\sectionmark}[1]{}
\renewcommand{\Im}{\operatorname{Im}}

\newcommand{\ve}{\varepsilon}

\renewcommand{\Im}{\operatorname{Im}}

\makeatletter
\newcommand{\Extend}[5]{\ext@arrow 0099{\arrowfill@#1#2#3}{#4}{#5}}
\makeatother

\begin{document}

\baselineskip20pt

 \title[Differential inequalities, normality and quasi-normality]
 {Differential inequalities, normality and quasi-normality}
\author[Xiaojun Liu, Shahar Nevo and Xuecheng Pang]{Xiaojun Liu, Shahar Nevo and Xuecheng
Pang}
\thanks{Research of first author supported by the NNSF of China Approved
No.11071074 and also supported by the Outstanding Youth Foundation
of Shanghai No. slg10015.}

 \thanks{Research of third author supported by the NNSF of China Approved
No.11071074.}

\address{Xiaojun Liu, Department of Mathematics,
University of Shanghai for Science and Technology, Shanghai 200093,
P.R. China} \email{Xiaojunliu2007@hotmail.com}

\address{Shahar Nevo, Department of Mathematics,
Bar-Ilan University, 52900 Ramat-Gan, Israel}
\email{nevosh@macs.biu.ac.il}

\address{Xuecheng Pang, Department of Mathematics,
East China Normal University, Shanghai 200241, P.R.China}
\email{xcpang@math.ecnu.edu.cn}

 %\thanks       {Research supported by the German-Israeli Foundation for Scientific
 %Research and Development, G.I.F. Grant No. I-809-234.6/2003}

\begin{abstract}
We prove that if $D$ is a domain in $\mathbb C$, $\alpha>1$ and
$C>0$, then the family $\mathcal F$ of functions $f$ meromorphic
in $D$ such that$$\frac{|f'(z)|}{1+|f(z)|^{\alpha}}>C\quad
\text{for every }z\in D$$is normal in $D$. For $\alpha=1$, the
same assumptions imply quasi-normality but not necessarily
normality.
\end{abstract}

 %{30A10, 30D45}

\maketitle

 \section{Introduction}\label{introduction}
  % \numberwithin{equation}{section}
Throughout we use the following notation $D$ denotes a domain in
$\mathbb C$. For $z_0\in\C$ and $r>0$,
$\Delta(z_0,r)=\{z:|z-z_0|<r\}$,
$\Delta'(z_0,r)=\{z:0<|z-z_0|<r\}$,
$\overline{\Delta}(z_0,r)=\{z:|z-z_0|\le r\}$,
$\Gamma(z_0,r)=\{z:|z-z_0|=r\}$ and
$R(z_0,R_1,R_2)=\{z:R_1<|z-z_0|<R_2\}$. We write
$f_n(z)\overset\chi \Rightarrow f(z)$ on $D$ to indicate that the
sequence $\{f_n\}$ converges to $f$  in the spherical metric,
uniformly on compact subsets of $D$, and $f_n\Rightarrow f$ on $D$
if the convergence is in the Euclidean metric. The spherical
derivative is denoted by $f^\#(z)$. We shall also use the notion
of $Q_m-$ normality. For this recall that given a set $E\subset
D$, then the derived set of order $m$ of $E$ with respect to $D$
is defined by induction: $E^{(1)}_D$ is the set of accumulation
points of $E$ in $D$. $E^{(m)}_D=\left(
E^{(m-1)}_D\right)^{(1)}_D$. A family $\mathcal F$ of functions
meromorphic in $D$ is said to be $Q_m-$normal in $D$ if every
subsequence $\{f_n\}^\infty_{n=1}$ of functions from $\mathcal F$
has a subsequence that converges uniformly with respect to $\chi$
on $D\backslash E$, where $E^{(m)}_D=\emptyset$ (Here if $m=0$,
then $\mathcal F$ is in fact normal family and if $m=1$, then
$\mathcal F$ is quasi-normal family). If, in addition there exists
some $\nu\in\mathbb N$, such that $E$ can always be taken to
satisfy $\left|E^{(m-1)}_D\right|\le \nu$, then $\mathcal F$ is
said to be $Q_m-$normal family of order at most $\nu$.

For more about $Q_m-$normality see \cite{1}. This paper deals with
the meaning of some differential inequalities. A natural point of
departure is the following famous criterion of normality due to F.
Marty.

\begin{MartyTheorem*}\cite[p.~75]{6}
A family $\mathcal F$ of functions meromorphic in a domain $D$ is
normal if and only if $\{f^\#(z):f\in\mathcal F\}$ is locally
uniformly bounded in $D$.
\end{MartyTheorem*}

Following Marty's Theorem, L. Royden proved the following generalization.

\begin{TheoremR*}\cite{5} Let $\mathcal F$ be a family of meromorphic
functions in $D$, with the property that for each compact set $K\subset D$,
there is a positive  increasing function $h_K$, such that
$$|f'(z)|\le h_K(|f(z)|)$$for all $f\in\mathcal F$ and $z\in K$, then
$\mathcal F$ is normal in $D$.
\end{TheoremR*}
This result was significantly extended further in various
directions, see \cite{3}, \cite{7} and \cite{9}.

\medskip

In \cite{2}, J. Grahl and the second author proved a counterpart to Marty's Theorem.
\begin{TheoremGN*}Let $\mathcal F$ be a family of functions meromorphic in $D$ and let
$\ve>0$. If $f^\#(z)\ge\ve$ for every $f\in\mathcal F$ and $z\in
D$, then $\mathcal F$ is normal in $D$.
\end{TheoremGN*}
It is equivalent to say that local uniform boundedness of the spherical
derivatives from zero implies normality.

\medskip

The proof uses mainly Gu's criterion to normality, Zalcman's Lemma
and Pang-Zalcman Lemma. N. Steinmetz \cite{8} gave shorter proof
of Theorem GN, using the Schwarzian derivative and some well-known
facts on linear differential equations. Here in this paper, we
prove a generalization of Theorem GN (with much simpler proof) and
also, for the first time we present a differential inequality that
distinguish between normality to quasi-normality.

\begin{theorem}\label{thm1} Let $0\le\alpha<\infty$ and $C>0$. Let
$\mathcal F_{\alpha,C}(D)$ be the family of all meromorphic functions $f$ in
$D$, such that$$\frac{|f'(z)|}{1+|f(z)|^{\alpha}}>C\quad
\text{for every }z\in D.$$The the following hold:
\begin{enumerate} \item[(1)] If $\alpha>1$, then $\mathcal F_{\alpha,C}(D)$
is normal in $D$;
\item[(2)] If $\alpha=1$, then $\mathcal F_{\alpha,C}(D)$ is quasi-normal in $D$, but
not necessarily normal.
\end{enumerate}
\end{theorem}

\medskip

In section 2, we prove Theorem \ref{thm1}. In section 3, we show that $\mathcal F_{1,C}(D)$ can be
of infinite order and discuss the validity of Theorem \ref{thm1} for $\alpha<1$. In section 4,
we discuss the reverse inequality $\DF{|f'(z)|}{1+|f(z)|^{\alpha}}<C$.

\section{Proof of Theorem \ref{thm1}}

We first state explicity the famous lemma of Pang and Zalcman (that was already mentioned).
Observe that this lemma is ``if and only if".

\begin{lemma}\label{lemma1}\cite{4} Let $\mathcal F$ be a family of meromorphic functions in
a domain $D$, all of whose zeros have multiplicity at least $m$, and all of whose poles
have multiplicity at least $p$, and let $-p<\alpha<m$. Then $\mathcal F$ is not normal at some
$z_0\in D$ if and only if there exist sequences $\{f_n\}^\infty_{n=1}\subset\mathcal
F$, $\{z_n\}^\infty_{n=1}\subset D$, $\{\rho_n\}^\infty_{n=1}\subset(0,1)$, such that
$\rho_n\to0^+$, $z_n\to z_0$ and$$g_n(\zeta):=\rho^{-\alpha}_nf_n(z_n+\rho_n\zeta)
\overset\chi\Rightarrow g(\zeta)\quad\text{on}\quad\mathbb C,$$where $g$ is a
nonconstant meromorphic function in $D$.
\end{lemma}

Here the ``if" direction holds if $g_n(\zeta)$ converges in some
open set $\Omega\subset\mathbb C$ to a nonconstant meromorphic
function $g$ in $\Omega$. For a full proof of this lemma see
\cite{2}.

\subsection{Proof of (1) of  Theorem \ref{thm1}}

Let $\{f_n\}^\infty_{n=1}$ be a sequence of functions in $\mathcal
F_{\alpha,C}(D)$. Let $z_0\in D$ and assume by negation that
$\{f_n\}$ is not normal at $z_0$. Suppose that there exist $r>0$,
such that each $f_n$ is holomorphic in $\Delta(z_0,r)$. We take
$\beta>\DF1{\alpha-1}>0$. By Lemma \ref{lemma1}, there exist a
subsequence of $\{f_n\}$, that without loss of generality will
also be denoted by $\{f_n\}^\infty_{n=1}$ and sequences
$\rho_n\to0^+$, $z_n\to z_0$, such that
\begin{equation}
\label{1}
g_n(\zeta):=\rho^{\beta}_nf_n(z_n+\rho_n\zeta)
\overset\chi\Rightarrow g(\zeta)\quad\text{on}\quad\mathbb C,
\end{equation}
where $g$ is a nonconstant entire function in $\mathbb C$.

Taking $\zeta_0\in\mathbb C$, such that
\begin{equation}
\label{2}
g(\zeta_0)\ne0,\ \infty.
\end{equation}
Then by  \eqref{1}, \eqref{2} and the value of $\beta$, we get that for large enough $n$, $$\begin{aligned}&\rho^{\beta+1}_n\left|f'_n(z_n+\rho_n\zeta_0)\right|=\rho^{1+\beta-\beta\alpha}_n\left|\frac{f'_n(z_n+\rho_n\zeta_0)}
{f^\alpha_n(z_n+\rho_n\zeta_0)}\right||\rho^\beta_nf_n(z_n+\rho_n\zeta_0)|^\alpha\\
&>\rho^{1+\beta-\beta\alpha}_n\cdot C\frac{|g(\zeta_0)|^\alpha}2\underset{n\to\infty}\longrightarrow\infty.
\end{aligned}$$
We thus got a contradiction and the holomorphic case is proven.

Suppose now that there is no $r>0$, such that for infinitely many
indices $n$, $f_n$ is holomorphic in $\Delta(z_0,r)$. Hence we
deduce the existence of some subsequence of
$\{f_n\}^\infty_{n=1}$, that without loss of generality will also
be denoted by $\{f_n\}^\infty_{n=1}$, and a sequence $z_n\to z_0$,
such that $f_n(z_n)=\infty$ (otherwise we are again in the
holomorphic case and we are done). We can also assume (after
moving to subsequence of $\{f_n\}^\infty_{n=1}$...) that there
exist a sequence $\widetilde{z}_n\to z_0$, such that
$f_n(\widetilde{z}_n)=0$. Indeed, otherwise, for some $\delta>0$
and large enough $n$, $f_n\ne0$ in $\Delta(z_0,\delta)$ and
$|f'_n|>C$ there. Then by Gu's criterion we deduce that $\{f_n\}$
is normal.
\begin{claim*}$\left\{\DF{f_n}{f'_n}\right\}^\infty_{n=1}$ is normal in $D$.
\end{claim*}

\textit{Proof of Claim}. If $|f_n(z)|\le1$, then $\left|\DF{f'_n(z)}{f_n(z)}\right|\ge|f'_n(z)|>C$.
If $|f_n(z)|>1$, then $\left|\DF{f'_n(z)}{f_n(z)}\right|\ge
\DF{|f'_n(z)|}{1+|f_n(z)|}>\DF{|f'_n(z)|}{1+|f_n(z)|^\alpha}>C$. Thus in any case, $\left|\DF{f'_n(z)}{f_n(z)}\right|>C$
for every $n$ and every $z\in D$. Hence $\left\{\DF{f'_n}{f_n}\right\}^\infty_{n=1}$ is normal and so is
$\left\{\DF{f_n}{f'_n}\right\}^\infty_{n=1}$.

\medskip

According to the claim, we can assume (after moving to subsequence
of $\{f_n\}^\infty_{n=1}$...) that
$$\frac{f_n(z)}{f'_n(z)}\overset\chi\Rightarrow H(z)\quad\text{in}\quad D.$$Since $\DF{f_n}{f'_n}$
vanish at the zeros and at the poles of $f_n$, we deduce that $H$ is holomorphic in $D$. We have
$$\left(\DF{f_n}{f'_n}\right)'=1-\frac{f_nf''_n}{f^{'2}_n}.$$Thus we have
$$\left(\DF{f_n}{f'_n}\right)'\Bigg|_{z=\widetilde{z}_n}=1.$$At the poles $z_n$ of $f_n$, the situation
is different. Each $z_n$ is a pole of order $k=k_n$ of $f_n$. This means that in some neighborhood of $z_n$, we have
$$f_n(z)=\frac{a_{-k}}{(z-z_n)^k}+\frac{a_{-k+1}}{(z-z_n)^{k-1}}+\cdots\quad
\quad(a_{-k}\ne0).$$Thus$$f'_n(z)=\frac{-ka_{-k}}{(z-z_n)^{k+1}}+\cdots,\quad\text{and}\quad f''_n(z)=\frac{k(k+1)a_{-k}}{(z-z_n)^{k+2}}+\cdots.$$
We then get that$$\frac{f_nf''_n}{f'^2_n}\Bigg|_{z=z_n}=\frac{k(k+1)a^2_{-k}}{(ka_{-k})^2}=\frac{k+1}k=1+\frac1k,$$and so
$$\left(\DF{f_n}{f'_n}\right)'\Bigg|_{z=z_n}=1-\left(1+\frac1k\right)=-\frac1k.$$Since $z_n\to z_0$ and also
$\widetilde{z}_n\to z_0$, we get a contradiction to any possible value of $H'(0)$. This completes the proof of (1).

\subsection{Proof of (2) of  Theorem \ref{thm1}}

The family $\{nz:\ n\in\mathbb N\}$ which is not normal at $z=0$, shows that local uniform boundedness of $\left\{\DF{|f'|}{1+|f|}:\ f\in\mathcal F\right\}$
does not imply in general normality. In order to prove quasi-normality, observe first that for every $f\in\mathcal F_{1,C}(D)$, we have
$\left|\DF{f'}{f}\right|>C$ and also $|f'|>C$. Thus both $\{f':\ f\in\mathcal F_{1,C}(D)\}$ and $\{f'/f:\ f\in\mathcal F_{1,C}(D)\}$ are
normal in $D$.

Let us take now a sequence $\{f_n\}^\infty_{n=1}$ of functions
from $\mathcal F_{1,C}(D)$. If, by negation $\{f_n\}_n$ is not
normal at some $\widehat{z}_0\in D$, then we can assume (after
moving to subsequence...) that there exist $z_n\to\widehat{z}_0$,
and $\rho_n\to0^+$ and a nonconstant function $g$, meromorphic in
$\mathbb C$ such that
$$g_n(\zeta)=\rho^{-\frac12}_nf_n(z_n+\rho_n\zeta)\overset{\chi}\Rightarrow g(\zeta)\quad\text{on}\quad\mathbb C.$$
Let $P_g$ denotes the set of poles of $g$ in $\mathbb C$. If $g$ is not of the form $g(\zeta)=a\zeta+b$, then we get by
differentiation,
$$g'_n(\zeta)=\rho^{\frac12}_nf'_n(z_n+\rho_n\zeta)\Rightarrow g'(\zeta)\quad\text{on}\quad\mathbb C\backslash P_g.$$
The derivative $g'$ is nonconstant, and thus by Lemma
\ref{lemma1}, $\{f'_n\}_{n=1}^\infty$ is not normal at
$\widehat{z}_0$, a contradiction.

\medskip

Thus, we must have $g(\zeta)=a\zeta+b$ $(a\ne0)$ and by Rouche's
Theorem, for any neighborhood $U$ of $\widehat{z}_0$, $f_n$ has
for large enough $n$ a zero in $U$. This means that we can assume
(after moving to subsequence...) that there exists a sequence
$z^\ast_n\to\widehat{z}_0$, such that $f_n(z^\ast_n)=0$.

Now, suppose by negation that $\{f_n\}^\infty_{n=1}$ is not
quasinormal at some $z_0\in D$. After moving to subsequence, that
will also be called $\{f_n\}^\infty_{n=1}$, we can assume that
there exist a sequence $\{z_k\}^\infty_{k=1}$ of distinct points
in $D$, such that $z_k\underset{k\to\infty}\longrightarrow z_0$
and each subsequence of $\{f_n\}^\infty_{n=1}$ is not normal at
each $z_k$. According to the previous discussion, for every
$k=1,2,\cdots$, there exists $n_k$ and a sequence
$\{z_{k,n}\}^\infty_{n=n_k}$,
$z_{k,n}\underset{n\to\infty}\longrightarrow z_k$, such that
$f_n(z_{k,n})=0$ for every $n\ge n_k$.

Hence for every $\delta>0$, and for every $N\in\mathbb N$, $f_n$
has in $\Delta(z_0,\delta)$ at least $N$ zeros for large enough
$n$. Now, since $\left\{\DF{f_n}{f'_n}:n\in\mathbb
N\right\}^\infty_{n=1}$ is normal, we can also assume (after
moving to subsequence...) that
$$\frac{f_n(z)}{f'_n(z)}\Rightarrow H(z)\quad\text{in}\quad D,$$where $H$ is holomorphic in $D$. Each zero of $f_n$ is also a zero of
$f_n/f'_n$, so by the above discussion the number of zeros of $f_n$ in any neighborhood of $z_0$ tends to $\infty$, as
$n\to\infty$, and thus we conclude that $H\equiv0$. Hence we have
$$\left(\frac{f_n}{f'_n}\right)'\Rightarrow0\quad\text{in}\quad D.$$
But on the other hand,
$$\left(\frac{f_n}{f'_n}\right)'\Bigg|_{z=z_{k,n}}=1.$$This is a contradiction and (2) of Theorem \ref{thm1} is proven.

\section{Some remarks}
\subsection{The order of quasi-normality of $\mathcal F_{1,C}(D)$}

We shall show now that the order of quasi-normality of $\mathcal F_{1,C}(D)$ can general be
large as we we like. Since we can make a linear change of the variable, it is enough if we construct in some
specific domain $D$, a sequence $\{f_n\}^\infty_{n=1}$ of functions such that every subsequence of  $\{f_n\}^\infty_{n=1}$
has the same infinite set of points of non-normality in $D$, and
$$\inf\limits_{z\in D}\frac{|f'_n(z)|}{1+|f_n(z)|}\ge C\quad\text{for some}\quad C>0.$$
So let $D=\left\{z:|\Im z|<1,\ |z-\pi k|>\DF12,\ k\in\mathbb Z\right\}$, and define for every $n\ge1$, $f_n(z)=n\cos z$.
It is obvious that every subsequence of $\{f_n\}^\infty_{n=1}$ is not normal exactly at the points $z_k=\DF\pi2+\pi k$, $k\in
\mathbb Z$. Thus $\{f_n\}^\infty_{n=1}$ is quasi-normal of infinite order in $D$. Because of the periodicity of $\cos z$, there
exist some $C>0$, such that
$$\frac{|f'_n(z)|}{1+|f_n(z)|}\ge C\quad\text{for every}\ n\ \text{and for every}\ z\in D.$$
Hence $\mathcal F_{1,C}(D)$ is quasi-normal of infinite order in
$D$. We deduce that for every domain $D$, and for every
$\nu\in\mathbb N$ there exists $C_{D,\nu}>0,$ such that $\mathcal
F_{1,C_{D,\nu}}(D)$ is quasi-normal   in $D$, but not quasi-normal
of order at most $\nu$.

\subsection{The case $0\le\alpha<1$}
In this case for every bounded domain $D$ and every $C>0$,
$\mathcal F_{\alpha,C}(D)$ has no degree of normality. To be more
precise we have the following theorem.
\begin{theorem}\label{thm3} Let $0\le\alpha<1$, $m\ge0$, $C>0$ and $D$ a bounded domain in $D$.
Then $\mathcal F_{\alpha,C}(D)$ is not $Q_m-$normal in $D$.
\end{theorem}
\begin{proof}
For a given $0\le\alpha<1$, let us first prove the theorem for
some specific domain. Let $1<\ve<3^{\frac{1-\alpha}{1+\alpha}}$.
Consider the polynomial functions $P_n(z)=z^n-3^n$ defined on the
ring $D_\ve:=R\left(0,\DF3\ve,3\ve\right)$. Clearly every
subsequence of $\{P_n\}^\infty_{n=1}$ is not normal exactly at any
point of $\Gamma(0,3)$ ($\Gamma(0,3)$ is of power $\aleph$ and of
course $\left(\Gamma(0,3)\right)^{(m)}_{D_\ve}=\Gamma(0,3)$ for
every $m\ge1$).

\begin{claim*}$\inf\limits_{z\in D_\ve}\DF{|P'_n(z)|}{1+|P_n(z)|^\alpha}\underset{n\to\infty}\longrightarrow\infty$.
\end{claim*}
\textit{Proof of Claim}. For every $z\in D_\ve$, we have
$$\frac{|P'_n(z)|}{1+|P_n(z)|^\alpha}=\frac{n|z|^{n-1}}{1+|z^n-3^n|^\alpha}>\frac{n\cdot\left(\frac3\ve\right)^n\cdot\frac\ve3}
{1+(2\cdot(3\ve)^n)^\alpha}>\frac{n\cdot\left(\frac3\ve\right)^n\cdot\frac\ve3}
{2(2\cdot(3\ve)^n)^\alpha}=\frac{n\ve}{6\cdot2^{1+\alpha}}\left(\frac{3^{1-\alpha}}{3^{1+\alpha}}\right)^n.$$
Since $\ve^{1+\alpha}<3^{1-\alpha}$, the last expression tends to
$\infty$, as $n\to\infty$, and this proves the claim.

\medskip

Now, give $C>0$, we have by the claim that there exists $N$, such that
$$\inf\limits_{z\in D_\ve}\DF{|P'_n(z)|}{1+|P_n(z)|^\alpha}>C\quad\text{for}\quad n\ge N,$$and thus
$\{P_n\}^\infty_{n=N}\subset\mathcal F_{\alpha,C}(D_\ve)$. Since $\{P_n\}^\infty_{n=N}$ is not $Q_m-$normal
in $D_\ve$, we proved the theorem for $D=D_\ve$.

\medskip

Now, let $D$ be some bounded domain. There is a ring $R(z_0,R_1,R_2)$, together with a linear transformation $\varphi$,
$\varphi(z)=az+b$, such that $\varphi:$ $R(z_0,R_1,R_2)\longrightarrow D_\ve$ is one to one and onto (that is, $R(z_0,R_1,R_2)$
and $D_\ve$ are conformally equivalent) and such that $\varphi^{-1}(\Gamma(0,3))\cap D$ contain an arc of a circle. Every
subsequence of $\{P_n\circ\varphi\}^\infty_{n=1}$ is not $Q_m-$normal in $D$, for every $m\ge1$. Also for every $C>0$, there
exists $N$, such that $\{P_n\circ\varphi\}^\infty_{n=N}$ is contained in  $\mathcal F_{\alpha,C}(D)$ and thus also
$\mathcal F_{\alpha,C}(D)$ is not $Q_m-$normal in $D$ for every $m\ge1$. This completes the proof of the theorem.
\end{proof}

\subsection{The case $\alpha<0$}

Consider the family $\mathcal F_{\alpha,C}(D)$. If $\alpha<0$ and $f(z)=0$, then $|f(z)|^\alpha$ is not well-defined,
so if we require in addition that $f\ne0$, then since $|f'|>C$, we get by Gu's criterion that
$\mathcal F_{\alpha,C}(D)$ is normal.

If we permit that $f(z)=0$, then consider $z_0$, such that $f(z_0)=0$, we have
$$\lim\limits_{z\to z_0}\frac{|f'(z)|}{1+|f(z)|^\alpha}=0,$$and so the condition $$\frac{|f'(z)|}{1+|f(z)|^\alpha}>C$$
cannot be satisfied.

\section{The reverse inequality $\DF{|f'|}{1+|f|^\alpha}<C$}

Let $\alpha>0$, and let $\mathcal F$ be a family of functions
meromorphic in a domain $D$. By Theorem R, if $\mathcal
F_\alpha:=\left\{\DF{f'}{1+|f|^\alpha}:f\in\mathcal F\right\}$ is
locally uniformly bounded in $D$, then $\mathcal F$ is normal.

\medskip

For $0\le\alpha\le1$, the converse is false. Consider the family $\mathcal F=\{z^n:n\in\mathbb N\}$, in
$D=\Delta(3,1)$. Obviously, $f_n(z)\Rightarrow\infty$ in $D$, but$$\frac{|f'_n(z)|}{1+|f_n(z)|^\alpha}=\DF{n|z|^{n-1}}{1+|z|^{n\alpha}}.$$
Thus, since $\alpha<1$, we get that$$\inf\limits_{z\in D}\frac{|f'_n(z)|}{1+|f_n(z)|^\alpha}\underset{n\to\infty}\longrightarrow
\infty,$$and thus $\mathcal F_\alpha$ is not locally uniformly bounded.

\medskip

For $\alpha\ge2$, the converse holds.

Indeed, assume that $\mathcal F$ is normal in $D$. We have for
every $f\in\mathcal F$ and $z\in D$
$$\frac{|f'(z)|}{1+|f(z)|^\alpha}=\frac{|f'(z)|}{1+|f(z)|^2}\cdot\frac{1+|f(z)|^2}{1+|f(z)|^\alpha}.$$
By Marty's Theorem, $\mathcal F_2$ is locally uniformly bounded in $D$. In addition,
$h(x)=\DF{1+x^2}{1+x^\alpha}$ is bounded in $[0,+\infty)$, and there is some $M>0$, such that
$$\DF{1+|f(z)|^2}{1+|f(z)|^\alpha}\le M$$ for every $f\in\mathcal F$, $z\in D$. We then deduce that
$\mathcal F_\alpha$ is locally uniformly bounded in $D$.

\medskip

We are left with the case $1<\alpha<2$. We show now that for
meromorphic functions, normality does not imply local uniform
boundedness, for every $1<\alpha<2$. Take $\mathcal
F=\left\{\DF1z\right\}$ (only a single function) in $\Delta$. We
have
$$\frac{|f'(z)|}{1+|f(z)|^\alpha}\underset{z\to0}\longrightarrow\infty.$$ For holomorphic functions, we can approve
the converse:

\begin{theorem}\label{thm4} Let $1<\alpha<2$. Suppose that $\mathcal F$ is a normal family of
holomorphic functions in $D$. Then
$$\mathcal F_{\alpha}=\left\{\frac{|f'(z)|}{1+|f(z)|^{\alpha}}:f\in\mathcal F\right\}$$
is locally uniformly bounded in $D$.
\end{theorem}
\begin{proof}
Suppose to the contrary that $\mathcal F_\alpha$ is not locally uniformly bounded in $D$. Then there exist $z_0\in D$, $z_n\to z_0$ and
$f_n\in\mathcal F$, such that
\begin{equation}
\label{10}
\frac{|f'_n(z_n)|}{1+|f_n(z_n)|^\alpha}\underset{n\to\infty}\longrightarrow\infty.
\end{equation}
The sequence $\{f_n\}^\infty_{n=1}$ has a uniform convergent subsequence in $D$, that without loss of generality
we also call  $\{f_n\}^\infty_{n=1}$. So we assume that
$$f_n\Rightarrow f\quad\text{in}\quad D.$$Let us separate into two cases, according to the behavior of $f$.

\medskip

\noindent\underbar{Case (1)} $f$ is holomorphic in $D$.

Then $f'_n\Rightarrow f'$ in $D$, and we easily get a contradiction to  \eqref{10}.

\medskip

\noindent\underbar{Case (2)} $f\equiv\infty$.

In particular, we have $$f_n(z_0)\underset{n\to\infty}\longrightarrow\infty.$$
We take $R>0$, such that $\overline{\Delta}(z_0,R)\subset D$ and
\begin{equation}
\label{11}
0<\rho<R\frac{\sqrt{1+\alpha}-\sqrt{2}}{\sqrt{1+\alpha}+\sqrt{2}}.
\end{equation}
By Harnack's inequality, for large enough $n$, we have for every $z\in\overline{\Delta}(z_0,\rho)$
\begin{equation}
\label{12}
|f_n(z_0)|^{\frac{R-\rho}{R+\rho}}\le|f_n(z)|\le|f_n(z_0)|^{\frac{R+\rho}{R-\rho}}.
\end{equation}
By \eqref{11} we get that
\begin{equation}
\label{13}
\frac{R+\rho}{R-\rho}<\sqrt{\frac{1+\alpha}2}\quad\bigg(\text{and
thus}\quad\frac{R-\rho}{R+\rho}>\sqrt{\frac2{1+\alpha}}\,\bigg).
\end{equation}
Now, by \eqref{12},\eqref{13} and Cauchy's integral formula, we get that for every $z\in\overline{\Delta}(z_0,\rho/2)$
and large enough $n$,
$$|f'_n(z)|=\frac1{2\pi}\left|\D\int\limits_{|\zeta-z_0|=\rho}\frac{f_n(\zeta)}{(\zeta-z)^2}d\zeta\right|\le
\frac{\rho}{(\rho/2)^2}\max\limits_{|\zeta-z_0|=\rho}|f_n(\zeta)|\le\frac4{\rho}|f_n(z_0)|^{\sqrt{\frac{1+\alpha}2}}.$$
Thus, by the last inequality, \eqref{11} and \eqref{12}, we have for large enough $n$,
$$\begin{aligned}\frac{|f'_n(z_n)|}{1+|f_n(z_n)|^\alpha}&\le\DF{\frac4{\rho}|f_n(z_0)|^{\sqrt{\frac{1+\alpha}2}}}{1+|f_n(z_0)|^{\frac\alpha{\sqrt{(1+\alpha)/2}}}}
\le\frac4\rho|f_n(z_0)|^{\sqrt{\frac{1+\alpha}2}-\frac\alpha{\sqrt{(1+\alpha)/2}}}\\
&=\frac4\rho|f_n(z_0)|^{\frac{1-\alpha}2\big/\sqrt{(1+\alpha)/2}}\underset{n\to\infty}\longrightarrow0.\end{aligned}$$
This is a contradiction to \eqref{10} and thus the Theorem
follows.
\end{proof}

\bibliographystyle{amsplain}

\end{document}